\documentclass[preprint,12pt]{elsarticle}

\usepackage{amsmath,amssymb,amsthm,amsfonts}
\usepackage{enumerate}
\usepackage{graphics}
\usepackage{graphicx}
\usepackage{url,color}

\newcommand{\R}{\mathbb{R}}
\newcommand{\IR}{\mathbb{IR}}

\newcommand{\inum}[1]{\mathbf{#1}}

\newtheorem{algorithm}{Algorithm}

\journal{Applied Mathematics and Computation}

\begin{document}
\begin{frontmatter}

\title{New parameterized solution
with application to bounding secondary variables in FE models of
structures}

\author[ep]{Evgenija D.\ Popova}
\ead{epopova@math.bas.bg}

\address[ep]{Institute of Mathematics and Informatics, Bulgarian Academy of Sciences \\ Acad. G.~Bonchev Str., block 8, 1113 Sofia, Bulgaria.}

\begin{abstract}
In this work we propose a new kind of parameterized outer estimate
of the united solution set to an interval parametric linear system.
The new method has several advantages compared to the methods
obtaining parameterized solutions considered so far. Some properties
of the new parameterized solution, compared to the parameterized
solution considered so far, and a new application direction are
presented and demonstrated by numerical examples. The new
parameterized solution is a basis of a new approach for obtaining
sharp bounds for derived quantities (e.g., forces or stresses) which
are functions of the displacements (primary variables) in interval
finite element models (IFEM) of mechanical structures.
\end{abstract}

\end{frontmatter}

\newtheorem{theorem}{Theorem}
\newtheorem{proposition}{Proposition}
\newtheorem{lemma}{Lemma}
\newtheorem{corollary}{Corollary}
\newtheorem{alg}{Algorithm}
\newtheorem{definition}{Definition}
\newtheorem{example}{Example}
\newtheorem{remark}{Remark}

\section{Introduction}

Denote by $\R^{m\times n}$  the set of real $m\times n$ matrices.
Vectors are considered as one-column matrices. A real compact
interval is $\inum{a} = [a^-, a^+] := \{a\in\R \mid a^-\leq a\leq
a^+\}$ and $\IR^{m\times n}$ denotes the set of interval $m\times n$
matrices. We consider systems of linear algebraic equations having
affine-linear uncertainty structure
\begin{equation} \label{pls}
\begin{split}
&A(p) x \; = \; a(p), \quad p\in\inum{p}\in\IR^K,\\ &A(p) := A_0 +
\sum_{k = 1}^K p_kA_k, \qquad  a(p)  := a_0 + \sum_{k = 1}^K p_ka_k,
\end{split}
\end{equation}
where $A_k\in\R^{n\times n}$, $a_k\in\R^n$, $k = 0,\ldots, K$ and
the parameters $p=(p_1,\ldots$, $p_K)^\top$ are considered to be
uncertain and varying within given non-degenerate\footnote{An
interval $\inum{a}=[a^-, a^+]$ is degenerate if $a^-=a^+$.}
intervals $\inum{p}=(\inum{p}_1,\ldots,\inum{p}_K)^\top$. The
so-called united parametric  solution set of the system (\ref{pls})
is defined by
$$\Sigma^p_{\rm uni}=\Sigma_{\rm uni}(A(p),a(p),\inum{p}) :=
\{x\in\R^n \mid (\exists p\in\inum{p})(A(p)x=a(p))\}.$$

Usually, the interval methods (designed to provide interval
enclosure of $\Sigma^p_{\rm uni}$) generate numerical interval
vectors that contain the solution set. A new type of solution,
$x(p,l)$, called parameterized or p-solution, providing outer
estimate of the united parametric solution set is proposed in
\cite{Kol2014}. This solution is in form of an affine-linear
function of interval-valued parameters
$$x(p,l) = \tilde{x} + Vp + l, \quad p\in\inum{p}, \; l\in\inum{l},$$
where $\tilde{x}\in\R^n$, $V\in\R^{n\times K}$ and $\inum{l}$ is an
$n$-dimensional interval vector. The parameterized solution has the
property $\Sigma^p_{uni}\subseteq x(\inum{p},\inum{l})$, where
$x(\inum{p},\inum{l})$ is the interval hull of $x(p,l)$ over
$p\in\inum{p}$, $l\in\inum{l}$. For a nonempty and bounded set
$\Sigma\subset\R^n$, its interval hull $\square\Sigma$ is defined by
$$\square\Sigma := \bigcap \{\inum{x}\in\IR^n \mid \Sigma\subseteq \inum{x}\}.$$
Since $x(p,l)$ is a linear function of interval parameters,
$$
x(\inum{p},\inum{l}) = \square \{x(p,l) \mid p\in\inum{p},
l\in\inum{l}\} = \tilde{x}+ V\inum{p} + \inum{l}.
$$

Parameterized forms of solution enclosures are proposed in relation
to different numerical methods, the latter yielding interval boxes
(vectors) containing the solution set, see, e.g.,
\cite{Kol2014}--\cite{SkalnaHla19} and the references in
\cite{SkalnaHla19} which mentions most of the works on parameterized
solution enclosures. Parameterized enclosure of parametric
$AE$-solution sets is developed in \cite{PopovaED:2017:POE}.  The
potential of the parameterized solution for solving some global
optimization problems where the parametric linear system (\ref{pls})
is involved as equality constraint is shown in \cite{Kol2016b}. All
parameterized solutions considered so far are functions of the
initial parameters $p$ of the system and of $n$  additional interval
parameters $l$, where $n$ is the dimension of the system. Therefore,
and in order to distinguish the newly proposed parameterized
solution, we will call all parameterized solutions considered so far
Kolev-style parameterized solutions, shortly $p,l$-solutions instead
of $p$-solutions.

In this work we propose a new parameterized outer estimate of the
united solution set to an interval parametric linear system. Basing
on a recently proposed framework for interval enclosure of the
united parametric solution set, which has a broader scope of
applicability \cite{PopovaED:2018}, the new parameterized method
has, respectively, a broader scope of applicability than most of the
methods obtaining parameterized solutions considered so far. For
parametric systems involving rank one uncertainty structure, the new
parameterized solution depends only on the initial parameters of the
system.

The structure of the paper is as follows. Section \ref{prelim}
introduces notation and known results about the parameterized
$x(p,l)$ solution. The new parameterized solution and its interval
enclosure property are derived in Section \ref{pNew}. Some geometric
properties of the parameterized solutions and theoretical comparison
between the two kinds of parameterized solutions are presented in
Section \ref{properties} along with numerical illustrative examples.
Section \ref{appl} presents a new application direction for the
newly proposed parameterized solution
--- a new simpler approach providing sharp bounds for derived variables in interval finite
element models (IFEM) of mechanical structures. The new
parameterized approach is illustrated by some example problems,
which demonstrate its ability to deliver sharp bounds to derived
variables with the same quality as those of the primary variables
with less effort. In these examples we compare the interval
enclosures obtained by the two kinds of parameterized solutions and
by the direct interval
approach, as well as by other approaches considered so far. %
The paper ends by some conclusions.

\section{Preliminaries} \label{prelim}

For $\inum{a} = [a^-, a^+]$, define its mid-point $\check{a}:= (a^-
+ a^+)/2$,  the radius $\hat{a} := (a^+ - a^-)/2$ and the magnitude
$|\inum{a}|:= \max\{|a^-|, |a^+|\}$. These functions are applied to
interval vectors and matrices componentwise. Inequalities are
understood componentwise. The spectral radius of a matrix
$A\in\R^{n\times n}$ is denoted by $\varrho(A)$. The identity matrix
of appropriate dimension is denoted by $I$. For $A_k\in\R^{n\times
m}$, $1\leq k\leq t$, $\left(A_1,\ldots,A_t\right)\in\R^{n\times
tm}$ denotes the matrix obtained by stacking the columns of the
matrices $A_k$. Denote the $i$-th column of $A\in\R^{n\times m}$ by
$A_{\bullet i}$ and its $i$-th row by $A_{i\bullet}$.

\begin{theorem}[{\cite{Rohn2011v22}, Theorem 4.4}] \label{RohnThm}
Let $\inum{A}=[I-\Delta, I+\Delta]\in\IR^{n\times n}$ with
$\Delta\in\R^{n\times n}$, $\varrho(\Delta)<1$. Then the inverse
interval matrix
$$\inum{A}^{-1}:= [\min\{A^{-1}\mid A\in\inum{A}\}, \max\{A^{-1}\mid A\in\inum{A}\}] = [\underline{H}, \overline{H}]$$
 is given by
\begin{eqnarray*}
\overline{H} = (\overline{h}_{ij}) &=& (I-\Delta)^{-1}, \\
\underline{H} = (\underline{h}_{ij}), & \; &
\underline{h}_{ij}=\begin{cases}
-\overline{h}_{ij} & \text{ if } i\neq j \\
\frac{\overline{h}_{jj}}{2\overline{h}_{jj}-1} & \text{ if }
i=j.\end{cases}
\end{eqnarray*}
\end{theorem}

Next we recall the simplest single step method for obtaining the
$p,l$-solution to a united parametric solution set $\Sigma^p_{\rm
uni}$. Since $[p^-,p^+] = \check{p}+[-\hat{p},\hat{p}]$, with  the
notation $A(\check{p}) = A_0+\sum_{i=1}^K \check{p}_i A_i$,
$a(\check{p}) = a_0+\sum_{i=1}^K \check{p}_i a_i$, system
(\ref{pls}) is equivalent to the interval parametric system
\begin{equation}\label{sysK2}
\left(A(\check{p})+\sum_{i=1}^K p_i'A_i\right)x = a(\check{p}) +
\sum_{i=1}^K p_i'a_i, \qquad p'\in [-\hat{p},\hat{p}]\in\IR^K.
\end{equation}
The following theorem is modified from \cite[Theorem 1]{Kol2016a}
for the system (\ref{sysK2}).

\begin{theorem}[{\cite[Theorem 1]{Kol2016a}}] \label{Kolev1}
Let $A(\check{p})$ in (\ref{sysK2}) be nonsingular. Denote
$\check{x}= \left(A(\check{p})\right)^{-1}a(\check{p})$, \
$F=(a_1,\ldots, a_K)$, \ $G = (A_1\check{x},\ldots, A_K\check{x})$,
\ $B^0 = \left(A(\check{p})\right)^{-1}(F-G)$. Assume that
\begin{equation}\label{rhoDelta}
\varrho
(\sum_{i=1}^K\left|\left(A(\check{p})\right)^{-1}A_i\right|\hat{p}_i)<1.
\end{equation}
Then
\begin{itemize}
\item[(i)] $A(p')$ in (\ref{sysK2}) is regular for each $p'\in [-\hat{p},\hat{p}]\in\IR^K$;

\item[(ii)] the united $p,l$-solution $x(p',l)$ of the system (\ref{sysK2})
exists and is determined by
\begin{equation}\label{x(p,l)}
x(p',l) = \check{x} + Vp' + l, \qquad p'\in
[-\hat{p},\hat{p}]\in\IR^K, \; l\in [-\hat{l},\hat{l}]\in\IR^n,
\end{equation}
\end{itemize}
where $V=\check{H}B^0$, $\hat{l}=\hat{H}|B^0|\hat{p}$, and
$\check{H}$, $\hat{H}$ are the midpoint and radius matrices,
respectively, of the inverse interval matrix
$\inum{H}=[\underline{H},\overline{H}]$ obtained by Theorem
\ref{RohnThm} for $\Delta =
\sum_{i=1}^K\left|\left(A(\check{p})\right)^{-1}A_i\right|\hat{p}_i$.
\end{theorem}

Most of the $p,l$-solutions considered so far
(\cite{Kol2014}--\cite{SkalnaHla19}, except \cite{Kolev2018})
require or check the condition (\ref{rhoDelta}), which determines
the scope of applicability of the Kolev-style parameterized
solutions. The method of Theorem \ref{Kolev1} will be used in the
comparisons that follow as a representative of all these methods,
although some of them may outperform others in the solution
enclosure.

\section{New parameterized solution for $\Sigma^p_{uni}$} \label{pNew}

Let ${\cal K}=\{1,\ldots,K\}$ and $\pi',\pi''$ be two subsets of
${\cal K}$ such that $\pi'\cap\pi''=\emptyset$, $\pi'\cup\pi''={\cal
K}$. Denote $p_\pi = (p_{\pi_1},\ldots,p_{\pi_{K_1}})$ for
$\pi\subseteq{\cal K}$, Card$(\pi)=K_1$. Denote by $D_{p_{\pi}}$ a
diagonal matrix with diagonal vector $p_{\pi}$.

In order to obtain a new parameterized solution to the united
parametric solution set of (\ref{pls}) we consider the following
equivalent form of the parametric system
\begin{equation}\label{plsG}
\left(A_0+LD_{g(p_{\pi'})}R\right)x = a_0+LD_{g(p_{\pi'})}t +
Fp_{\pi''}, \quad p\in\inum{p}
\end{equation}
with particular $\pi',\pi''\subseteq{\cal K}$ and suitable numerical
matrices $L,R$, numerical vector $t$, and a parameter vector
$g(p_{\pi'})$, which provide equivalent optimal rank one
representation (cf.~\cite{PopovaED:2018} or Definition
\ref{optimal}) of either $A(p_{\pi'})-A_0$, or of
$A^\top(p_{\pi'})-A_0^\top$, and $\sum_{k\in\pi'} p_ka_k =
LD_{g(p_{\pi'})}t$. The permutation $\pi'$ denotes the indices of
the parameters that appear in both the matrix and the right-hand
side of the system, while $\pi''$ involves the indices of the
parameters that appear only in $a(p)$ in (\ref{pls}). Next
definition is summarized from \cite{PopovaED:2018}.

\begin{definition}\label{optimal}
For a parametric matrix $A(p_{\pi'})=A_0+\sum_{k\in\pi'}p_kA_k$,
Card$(\pi')=K_1$, the following representation (called also
LDR-representation)
\begin{equation}\label{LDRmat}
A_0 + LD_{g(p_{\pi'})}R,
\end{equation}
where $g(p_{\pi'})\in\R^s$, $s=\sum_{k=1}^{K_1} s_k$, $g(p_{\pi'}) =
\left(g_1^\top (p_{\pi'_1}),\ldots, g_{K_1}^\top
(p_{\pi'_{K_1}})\right)^\top$, $L=\left(L_1,\ldots,
L_{K_1}\right)\in\R^{n\times s}$,
$R=\left(R_1^\top,\ldots,R_{K_1}^\top\right)^\top\in\R^{s\times n}$
and for $1\leq k\leq K_1$,
$g_k(p_{\pi'_k})=(p_{\pi'_k},\ldots,p_{\pi'_k})^\top\in\R^{s_k}$,
$p_{\pi'_k}A_{\pi'_k} = L_kD_{g_k(p_{\pi'_k})}R_k$,  is an {\bf
equivalent optimal rank one representation} of $A(p_{\pi'})$ if
\begin{itemize}
\item[(i)] (\ref{LDRmat}) restores $A(p_{\pi'})$ exactly, that is
$$
A(p_{\pi'}) = A_0 + LD_{g(p_{\pi'})}R =  A_0+ \sum_{k=1}^{K_1}
L_kD_{g_k(p_{\pi'_k})}R_k;
$$
\item[(ii)] for each parameter $g_i$,
$1\leq i\leq s$, $g_i\in g_i(\inum{p}_{\pi'})$, its coefficient
matrix $A_i$ has rank one, that is $A_i = L_{\bullet
i}R_{i\bullet}$;
\item[(iii)] for each $1\leq k\leq K_1$, the dimension $s_k$ of the diagonal
vector $g_k(p_{\pi'_k})$ is equal to the rank of $A_k$.
\end{itemize}
\end{definition}
It should be noted that condition (ii) of Definition \ref{optimal}
entails the adjective ``rank one'' of the
 representation (\ref{LDRmat}), while the condition (iii) entails the adjective
``optimal''. There are various ways to obtain the representation
(\ref{plsG}), cf.~\cite{Piziak}, \cite{PopovaED:2017:ESS}. In what
follows, in the representation (\ref{plsG}) we will not distinguish
between the equivalent representations and between the
representations originated from $A(p_{\pi'})$ or from $A^\top
(p_{\pi'})$; the difference is essential for the applications,
cf.~\cite[Example 8]{PopovaED:2017:ESS}. The following theorem
presents a method (proposed in \cite{PopovaED:2018}) for computing
numerical interval enclosure of a parametric united solution set.

\begin{theorem}\label{solution1}
Let the system (\ref{pls}) have equivalent representation
(\ref{plsG}) with optimal rank one representation of $A(p)$ and let
the matrix $A(\check{p})$ be nonsingular. Denote
$C=A^{-1}(\check{p})$ and $\check{x}=Ca(\check{p})$. If
\begin{equation}\label{strReg2}
\varrho\left(\left|(RCL)D_{g(\check{p}_{\pi'}-\inum{p}_{\pi'})}\right|\right)<1,
\end{equation}
then
\begin{itemize}
\item[(i)] $\Sigma_{uni}\left(A(p),a(p), \inum{p}\right)$ and the united
solution set $\Sigma_{\rm uni}({\rm Eq.}(\ref{eqY}))$ of the
interval parametric system
\begin{equation}\label{eqY}
\left(I-RCLD_{g(p_{\pi'})}\right)y = R\check{x} - RCFp_{\pi''} -
RCLD_{g(p_{\pi'})}t, \quad  p\in [-\hat{p},\hat{p}]
\end{equation}
are bounded,
\item[(ii)] $\inum{y}\supseteq\Sigma_{\rm uni}({\rm Eq.}(\ref{eqY}))$
is computable by methods that require (\ref{rhoDelta}),
\item[(iii)]
every $x\in\Sigma_{uni}\left(A(p), a(p), \inum{p}\right)$ satisfies
\begin{equation}\label{xx}
x  \in \check{x} - (CF)[-\hat{p}_{\pi''}, \hat{p}_{\pi''}] +
(CL)\left(D_{g([-\hat{p}_{\pi'},
\hat{p}_{\pi'}])}|\inum{y}-t|\right).
\end{equation}
\end{itemize}
\end{theorem}

\begin{theorem}\label{pg-sol}
Let the system (\ref{pls}) have equivalent representation
(\ref{plsG}) with optimal rank one representation of $A(p)$ and let
the matrix $A(\check{p})$ be nonsingular. Denote
$C=A^{-1}(\check{p})$ and $\check{x}=Ca(\check{p})$. If
(\ref{strReg2}) holds true, then
\begin{itemize}
\item[i)] there exists a united {\bf parameterized solution} of the system
(\ref{pls}), respectively the system (\ref{plsG}),
\begin{multline}\label{x(pg)}
x(p_{\pi''},g) = \check{x} - (CF)p_{\pi''} + \left(CLD_{|\inum{y} -
t|}\right)g,
\\ p_{\pi''}\in [-\hat{p}_{\pi''}, \hat{p}_{\pi''}], \; g\in
g([-\hat{p}_{\pi'}, \hat{p}_{\pi'}]),
\end{multline}
where $\inum{y}\supseteq\Sigma_{\rm uni}({\rm Eq.}(\ref{eqY}))$,
\item[ii)] with the same $\inum{y}$ used in (\ref{xx}) and in (\ref{x(pg)}),
the interval vector $x\left([-\hat{p}_{\pi''},
\hat{p}_{\pi''}],\right.$ $\left.g([-\hat{p}_{\pi'},
\hat{p}_{\pi'}])\right)$ is equal to the interval vector obtained by
Theorem \ref{solution1}.
\end{itemize}
\end{theorem}
\begin{proof}
Since (\ref{strReg2}) holds true, we apply Theorem \ref{solution1}
and solving (\ref{eqY}) obtain interval vector $\inum{y}\supseteq
\Sigma_{\rm uni}({\rm Eq.}(\ref{eqY}))$. By (iii) of Theorem
\ref{solution1}
 we obtain
\begin{equation*} 
\inum{x} =  \check{x} - (CF)[-\hat{p}_{\pi''}, \hat{p}_{\pi''}] +
(CL)D_{g([-\hat{p}_{\pi'}, \hat{p}_{\pi'}])}\left|\inum{y} -
t\right|.
\end{equation*}
Since this $\inum{x}$ is obtained as natural interval
extension\footnote{For interval extensions of a real function see,
e.g., \cite{Kearfott}.} of the function
\begin{equation}\label{x(p')}
x(p) = \check{x} - (CF)p_{\pi''} + (CL)D_{\left|\inum{y} -
t\right|}g(p_{\pi'}), \quad p\in [-\hat{p}, \hat{p}],
\end{equation}
we rename all interval parameters $p_{\pi'}$, that occur more than
once, and consider $x(p_{\pi''},g)$  in (\ref{x(pg)}) as a linear
interval function with single occurrence of the interval parameters
that satisfies (ii). Thus, the existence of (\ref{x(pg)}) and (ii)
follow from Theorem \ref{solution1}.
\end{proof}

It is clear from (\ref{x(pg)}) that the newly proposed parameterized
solution $x(p_{\pi''},g)$ is a linear function of Card$(\pi'')+s$
interval parameters $p_{\pi''}$, $g$. More precisely, this
parameterized solution is a function of $K+(s-K_1)$ interval
parameters $p$, $g'$, and the vector $g$ involves $s-K_1$ auxiliary
interval parameters $g'$. For the applications, it should be kept in
mind that it is actually a function of interval parameters $p\in
[-\hat{p}, \hat{p}]$ only, some of them multiply occurring in the
general case, see Section \ref{appl}. Representation (\ref{LDRmat})
is used in proving the regularity condition (\ref{strReg2}) and in
Theorems \ref{solution1}, \ref{pg-sol} above to approximate the
solution set of (\ref{pls}) by the solution set of a system
involving $s$ independent interval parameters in the matrix. More
details about the difference between the interval parametric
methods, based on condition (\ref{rhoDelta}), and those based on
rank one approximation of the parameter dependencies, condition
(\ref{strReg2}), can be found in \cite{PopovaED:2018}.

Parametric linear systems involving rank one interval
parameters are widely spread in various application domains.
Examples of such systems originating from models of electrical
circuits, in biology and structural mechanics are presented in
\cite{Pop2014CAMWA}.

\begin{corollary}\label{p-sol}
Let the system (\ref{pls}) have equivalent representation
(\ref{plsG}) with optimal rank one representation of $A(p)$ and each
$A_k$ have rank one. Let the matrix $A(\check{p})$ be nonsingular.
Denote $C=A^{-1}(\check{p})$ and $\check{x}=Ca(\check{p})$. If
(\ref{strReg2}) holds true, then
\begin{itemize}
\item[i)] there exists a united {\bf parameterized solution} of the system
(\ref{pls}), respectively the system (\ref{plsG}),
\begin{equation}\label{x(p)1}
x(p) = \check{x} - (CF)p_{\pi''} + \left(CLD_{|\inum{y} -
t|}\right)p_{\pi'}, \qquad p\in [-\hat{p}, \hat{p}],
\end{equation}
where $\inum{y}\supseteq\Sigma_{\rm uni}({\rm Eq.}(\ref{eqY}))$,
\item[ii)] with the same $\inum{y}$ used in (\ref{xx}) and in (\ref{x(p)1}),
the interval vector $x([-\hat{p},\hat{p}])$ is equal to the interval
vector obtained by Theorem \ref{solution1}.
\end{itemize}
\end{corollary}

The parameterized solutions (\ref{x(p,l)}), (\ref{x(pg)}),
(\ref{x(p)1}), as well as any other parameterized solutions of
$\Sigma^p_{\rm uni}$ considered so far, can be represented in a
uniform way as
\begin{equation*}
x(q) = \check{x} + U q, \quad q\in [-\hat{q}, \hat{q}],
\end{equation*}
where the parameter vector $q$ and the numerical matrix $U$ provide
equivalence to (\ref{x(p,l)}), (\ref{x(pg)}), (\ref{x(p)1}) of the
corresponding theorem, respectively. For example, $U=\left(-CF,
CLD_{|\inum{y}-t|}\right)$ and $q=\left(p_{\pi''}^\top,
g^\top\right)^\top$ give (\ref{x(pg)}) of Theorem \ref{pg-sol}.
Similarly,  $U=\left(V, I\right)$ and $q=\left((p')^\top,
l^\top\right)^\top$ give (\ref{x(p,l)}) of Theorem \ref{Kolev1},
while  $U=\left(CLD_{|\inum{y}-t|}, -CF\right)$ and
$q=\left(p_{\pi'}^\top, p_{\pi''}^\top\right)^\top$ give
(\ref{x(p)1}) of Corollary \ref{p-sol}.

\section{Properties and Comparison} \label{properties}

In this section we present some properties of the parameterized
solutions and compare the two kinds of these solutions.

\begin{theorem}\label{convex}
Geometrically, the two kinds of parameterized solutions, Kolev-style
$p,l$-solutions and the newly proposed $p,g$-solution, are bounded
convex polytopes.
\end{theorem}
\begin{proof}
From the representations (\ref{x(pg)}) and (\ref{x(p,l)}), it is
obvious that the two kinds of parameterized solutions are convex
polytopes as affine images of the interval boxes
$\left([-\hat{p}_{\pi''}, \hat{p}_{\pi''}]^\top, g([-\hat{p}_{\pi'},
\hat{p}_{\pi'}])^\top\right)^\top$ for $x(p_{\pi''},g)$ and
$([-\hat{p},\hat{p}]^\top, [-\hat{l}, \hat{l}]^\top)^\top$ for
$x(p,l)$. The convex polytopes are bounded due to the regularity
conditions (\ref{strReg2}) and  (\ref{rhoDelta}), respectively.
\end{proof}

The first difference between the two kinds of parameterized
solutions follows from the conditions (\ref{rhoDelta}),
(\ref{strReg2}) for their existence, which imply their scope of
applicability. It is proven in \cite[Theorem 3.2]{PopovaED:2018}
that condition (\ref{strReg2}) is more general than (\ref{rhoDelta})
and more powerful for large class of problems. Therefore, the newly
proposed parameterized solution $x(p_{\pi''},g)$ is applicable to a
wider class of parametric interval linear systems. The expanded
scope of applicability is demonstrated in \cite[Examples 5 and
8]{PopovaED:2017:ESS}, as well as in \cite{NePow}.  In what follows
we will not consider examples for which Kolev-style $p,l$-solutions
cannot be found. The focus will be on comparing the two kinds of
parameterized solutions when both exist.

\begin{theorem}\label{inclusion}
For a system (\ref{pls}) involving only rank one interval
parameters\footnote{rank$(A_k)=1$, $k=1,\ldots, K_1$} in the matrix
and for which both (\ref{rhoDelta}), (\ref{strReg2}) hold true, the
convex polytope representing a Kolev-style $p,l$-solution,
$\inum{l}\neq 0$, contains the convex polytope representing the
newly proposed $p$-solution.
\end{theorem}
\begin{proof} It follows from Corollary \ref{p-sol} and the assumptions of this theorem that the newly proposed parameterized
solution is function of less number of interval parameters. %
Since all vertices of the box $\inum{p}$ are vertices of the box
$\left(\inum{p}^\top, \inum{l}^\top\right)^\top$, the proof follows
from the properties of linear transformations, which transform the
vertices of a convex set into the same number of vertices of another
convex set. The inclusion will be more pronounced if $x(\inum{p},
\inum{l})\supset x(\inum{p})$.
\end{proof}

Our first example demonstrates Theorem \ref{inclusion} on a
parametric system for which the interval enclosures of the united
parametric solution set, obtained by the corresponding numerical
methods, are the same.

\begin{example}
Consider the interval parametric linear system
\begin{equation}\label{pls-ex1}
\begin{pmatrix}-1+\frac{1}{2}p_2 & -1 - \frac{1}{2}p_2 \\  -1 - p_2  &  -1 + p_2\end{pmatrix}x =
\begin{pmatrix}2+ p_2\\ -2p_2+3p_1\end{pmatrix}, \quad
p_1\in [-\frac{1}{4},1], p_2\in [\frac{1}{2}, \frac{3}{2}].
\end{equation}
For this system both conditions (\ref{rhoDelta}) and (\ref{strReg2})
are satisfied. Also, the two numerical methods (Theorem \ref{Kolev1}
and Theorem \ref{solution1}) yield the same interval vector
\begin{equation} \label{numX}
\inum{x} = \left([-\frac{17}{12}, \frac{55}{24}], [-\frac{27}{8},
-\frac{11}{12}]\right)^\top
\end{equation}
containing the united parametric solution set.

The  $p,l$-solution, obtained by Theorem \ref{Kolev1}, is
\begin{multline} \label{pl-ex1}
x'(p,l) = \begin{pmatrix}\frac{7}{16} \\
-\frac{103}{48}\end{pmatrix}
+ \begin{pmatrix}-\frac{27}{16}p_1 -\frac{21}{64}p_2 \\
\frac{9}{16} p_1 +
\frac{21}{64}p_2\end{pmatrix} + \begin{pmatrix}\frac{61}{96}l_1 \\
\frac{137}{192}l_2\end{pmatrix}, \\
p_1\in [-\frac{5}{8},\frac{5}{8}], p_2\in [-\frac{1}{2},
\frac{1}{2}], l_1,l_2\in [-1,1].
\end{multline}
Its numerical evaluation $x'(\inum{p},\inum{l})$ gives (\ref{numX}).

In order to obtain the newly proposed parameterized solution we
first obtain the equivalent form (\ref{plsG}) of the parametric
system (\ref{pls-ex1})
\begin{equation*}
(\check{A} + LD_{p_2}R)x = \check{a} + F(p_1) + LD_{p_2}t, \quad
p_1\in [-\frac{5}{8},\frac{5}{8}], p_2\in [-\frac{1}{2},
\frac{1}{2}],
\end{equation*}
where
\begin{multline*}
\check{A}=\check{p}_2\begin{pmatrix}\frac{1}{2} & - \frac{1}{2} \\
-1& 1\end{pmatrix}, \; L=\begin{pmatrix}\frac{1}{2}
\\ -1\end{pmatrix}, \; R=(1,-1), \; D_{p_2}=(p_2), \\
 \check{a} = \begin{pmatrix}2+ \check{p}_2\\ -2\check{p}_2+3\check{p}_1\end{pmatrix},\; F= \begin{pmatrix}0\\ 3\end{pmatrix}, \; t=(2).
\end{multline*}
The coefficient matrix of the parameter $p_2$ in (\ref{pls-ex1}) has
rank one. The interval parametric equation (\ref{eqY}) has the form
$$(1-p_2)y = \frac{31}{12} +2p_1-2p_2, \quad
p_1\in [-\frac{5}{8},\frac{5}{8}], p_2\in [-\frac{1}{2},
\frac{1}{2}]. $$ An interval enclosure of the solution set of the
last equation is
$$\inum{y} = [-\frac{1}{2}, \frac{17}{3}].$$
Then, by Corollary \ref{p-sol}, the parameterized solution is
\begin{equation} \label{p-ex1}
x''(p)= \begin{pmatrix}\frac{7}{16} \\ -\frac{103}{48}\end{pmatrix}
+
\begin{pmatrix}\frac{3}{2}p_1 +\frac{11}{6}p_2 \\ -\frac{1}{2} p_1 -
\frac{11}{6}p_2\end{pmatrix}, \quad p_1\in
[-\frac{5}{8},\frac{5}{8}], p_2\in [-\frac{1}{2}, \frac{1}{2}].
\end{equation}
Its interval evaluation $x''(\inum{p})$ gives also (\ref{numX}).
However, (\ref{p-ex1}) is a  $2$-polytope (in particular skew-box),
with a much smaller volume than the polytope of the $p,l$-solution
(\ref{pl-ex1}), both presented in Figure \ref{fig1}
\begin{figure}[ht]
\centerline{\includegraphics[scale=.8]{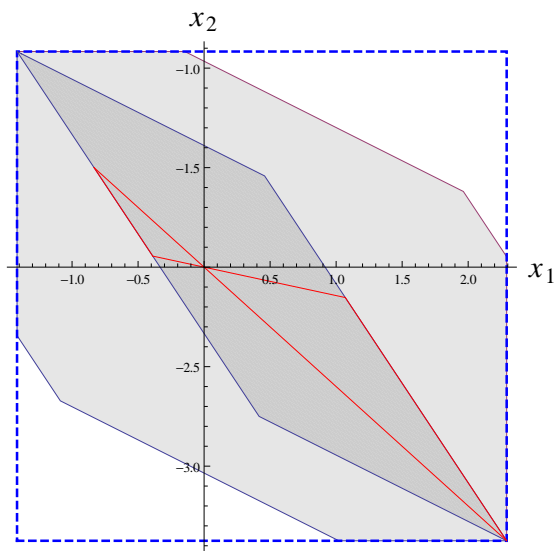}} 
\caption{The united parametric solution set of the system
(\ref{pls-ex1}) (the most inner butterfly region with red boundary),
its interval enclosure (\ref{numX}) (dashed line box), the
$p,l$-solution (light gray polytope) and the newly proposed
p-solution (dark gray polytope).}\label{fig1}
\end{figure}
\end{example}

\begin{example}
Consider the interval parametric linear system
\begin{multline}\label{pls-ex2}
\begin{pmatrix}-1+\frac{1}{2}p_2 -2p_3 & -1 - \frac{1}{2}p_2 \\  -1 - p_2  &  -1 + p_2\end{pmatrix}x =
\begin{pmatrix}p_2+3p_1-1\\ -2p_2+2p_1+3\end{pmatrix}, \\
p_1\in [-\frac{1}{4},1], \; p_2\in [\frac{1}{2}, \frac{3}{2}], \;
p_3\in [\frac{1}{5}, \frac{2}{3}].
\end{multline}
For this system both conditions (\ref{rhoDelta}) and (\ref{strReg2})
are satisfied. The method from Theorem \ref{Kolev1} yields interval
vector
$$\left([-5.725, 3.975], [-9.0805\ldots 5, 9.175]\right)^\top$$
and a parameterized solution enclosure
\begin{multline*}
x'(p,l)=\begin{pmatrix}-\frac{7}{8}\\\frac{17}{360}\end{pmatrix} +
\begin{pmatrix}-\frac{11881 p_1}{8400} + \frac{3124703 p_2}{1512000}\\-\frac{3969 p1}{2200} -\frac{168057 p_2}{44000}
+ \frac{1701 p_3}{880}\end{pmatrix} +
\begin{pmatrix}\frac{1108559}{378000}l_1\\
\frac{1116613}{198000}l_2\end{pmatrix}, \\
p_1\in [-\frac{5}{8},\frac{5}{8}], \; p_2\in [-\frac{1}{2},
\frac{1}{2}], \; l_1,l_2\in [-1,1].
\end{multline*}
The equivalent optimal rank one representation of system
(\ref{pls-ex2}) is obtained for
\begin{multline*}
L=\begin{pmatrix}1 &\frac{1}{2}\\ 0 &-1\end{pmatrix}, \; R=\begin{pmatrix}-2 &0\\ 1 &-1\end{pmatrix}, \; D_{p}=\begin{pmatrix}p_3 &0\\ 0 &p_2\end{pmatrix}, \\
 F= \begin{pmatrix}3\\ 2\end{pmatrix}, \; t=(0,2)^\top.
\end{multline*}
The interval parametric equation (\ref{eqY}) has the form
\begin{multline*}
\begin{pmatrix}1 &p_2\\ -\frac{2}{3}p_3 &1-\frac{58}{45}p_2 \end{pmatrix}y =
\begin{pmatrix}\frac{7}{4}-2p_1+2p_2\\ -\frac{83}{90} -\frac{4}{45}p_1 -\frac{116}{45}p_2\end{pmatrix},
\\
p_1\in [-\frac{5}{8},\frac{5}{8}], \; p_2\in [-\frac{1}{2},
\frac{1}{2}], \; p_3\in [-\frac{7}{30}, \frac{7}{30}].
\end{multline*}
An interval enclosure of the solution set of the last equation is
$$\inum{y} = \left([-5.7, 9.2], [-10.4, 77/9]\right)^\top.$$
Then, by Corollary \ref{p-sol}, the parameterized solution is
\begin{multline*} 
x''(p)= \begin{pmatrix}-\frac{7}{8} \\ \frac{17}{360}\end{pmatrix} +
\begin{pmatrix}1 \\ \frac{49}{45}\end{pmatrix}p_1 +
\begin{pmatrix}0 & \frac{31}{5}\\
-\frac{92}{15} & -\frac{2201}{225}
\end{pmatrix}\begin{pmatrix}p_3\\ p_2\end{pmatrix}, \\
p_1\in [-\frac{5}{8},\frac{5}{8}], \; p_2\in [-\frac{1}{2},
\frac{1}{2}], \;  p_3\in [-\frac{7}{30}, \frac{7}{30}].
\end{multline*}
The two parameterized solutions and their interval hulls are
presented and compared in Figure \ref{fig2}.
\begin{figure}[ht]
\includegraphics[scale=.7]{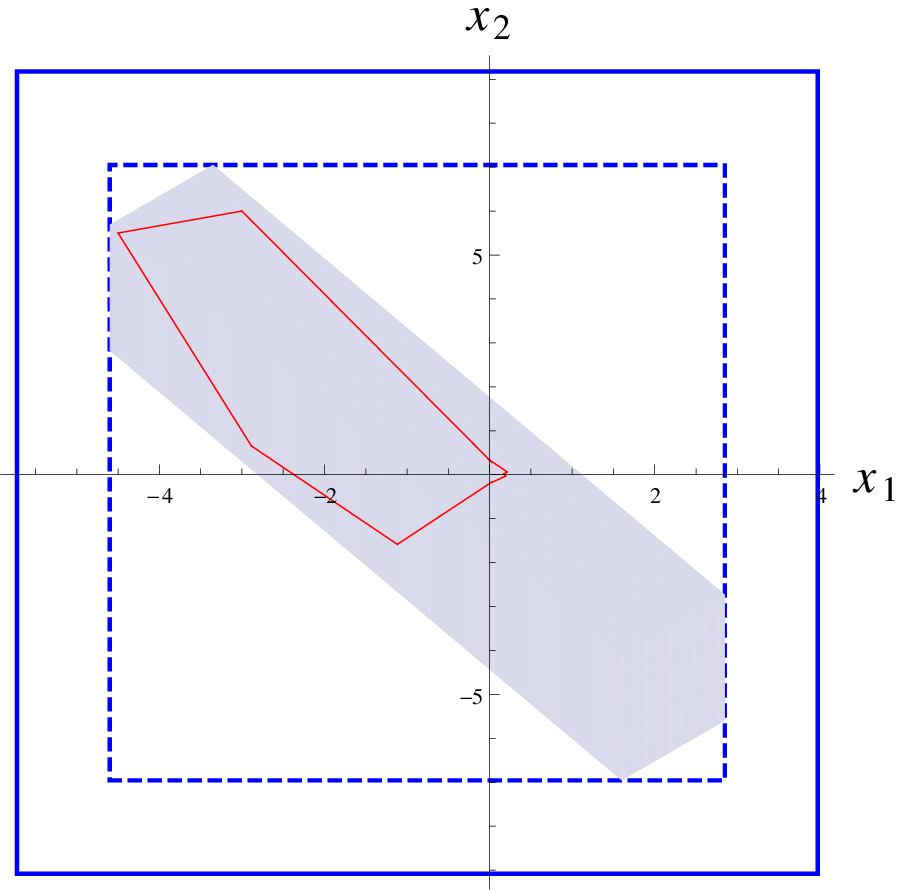}
\qquad
\includegraphics[scale=.7]{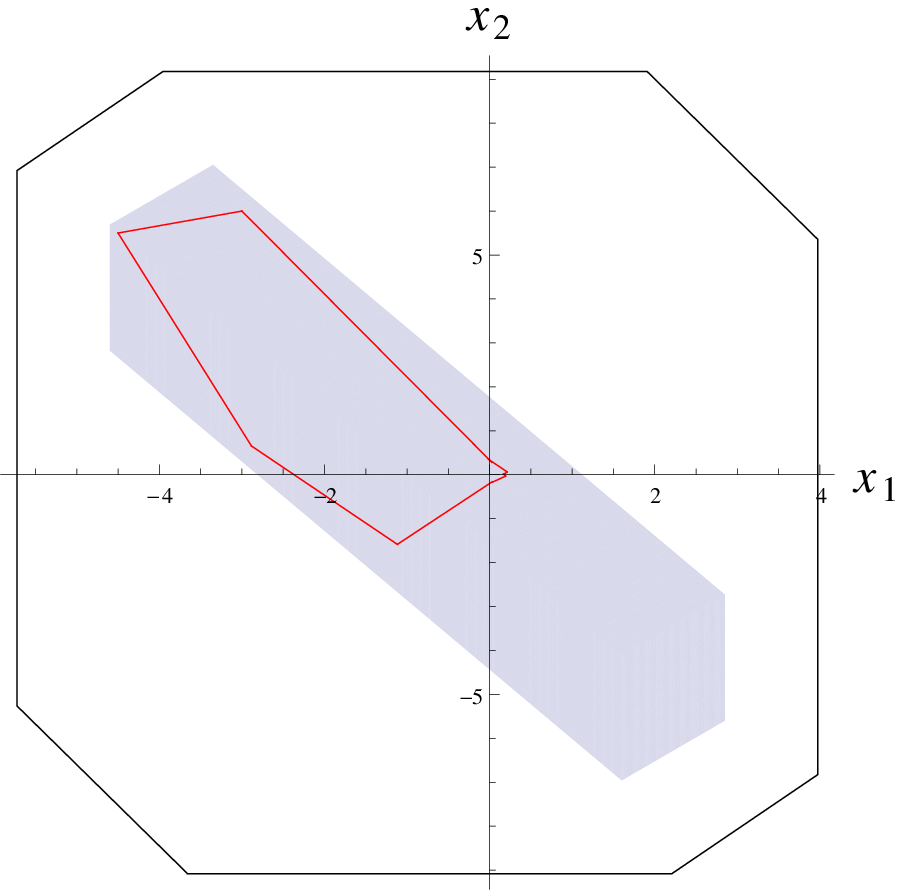}
\caption{Left: The united parametric  solution set of the system
(\ref{pls-ex2}) (the most inner region with red boundary), its
interval enclosure by Theorem \ref{solution1} (dashed line box), and
the interval enclosure by Theorem \ref{Kolev1} (the outer solid line
box). Right: The united parametric solution set of the system
(\ref{pls-ex2}) (the most inner region with red boundary), the
parameterized solution $x''(p)$ in gray and the parameterized
solution $x'(p,l)$ represented by its boundary.}\label{fig2}
\end{figure}
\end{example}

In general, when comparing the two parameterized solutions $x(p,l)$
and $x(p_{\pi''},g)$, one has to consider the two relations $K+n
\lesseqqgtr K+s-K_1$ and $x(\inum{p},\inum{l})\cong
x(\inum{p}_{\pi''},\inum{g})$, where $\sim\in\{\subset,\supset\}$.

\begin{example}\label{Ex3}
Consider the interval parametric linear system
\begin{multline}\label{pls-ex3}
\begin{pmatrix}\frac{1}{2}- p_2 & p_2 &2\\ p_2 & -p_2 & p_1\\
2 & p_1 & -2+p_1\end{pmatrix}x =
\begin{pmatrix}p_2\\ 2-p_1-p_2\\ p_1-1\end{pmatrix}, \\
p_1\in [\frac{2}{3},\frac{4}{3}], \; p_2\in [\frac{1}{2},
\frac{3}{2}].
\end{multline}
For this system both conditions (\ref{rhoDelta}) and (\ref{strReg2})
are satisfied. The interval hull of the united parametric solution
set, rounded outwardly and presented by $6$ digits in the mantissa,
is
$$\left([-.156997, .363637], [-.727273, .5972697], [.1896562, .4927185]\right)^\top.$$
The method from Theorem \ref{Kolev1} yields interval vector
$$\left([-.782941, .782941], [-1.014773, 1.6814392], [.082439, .584226]\right)^\top$$
and a parameterized solution enclosure
\begin{multline}\label{ex3'}
x'(p,l)=\begin{pmatrix}0\\\frac{1}{3}\\\frac{1}{3}\end{pmatrix} +
\begin{pmatrix}-0.38474 p_1 - 0.256493 p_2 \\
0.924365 p_1 + 0.728287 p_2\\ -0.392728 p_1 + 0.0374027
p_2\end{pmatrix} +
\begin{pmatrix}0.526447 l_1\\ 0.67584 l_2\\ 0.101283 l_3\end{pmatrix}, \\
p_1\in [-\frac{1}{3},\frac{1}{3}], \; p_2\in [-\frac{1}{2},
\frac{1}{2}], \; l_1,l_2,l_3\in [-1,1].
\end{multline}
The coefficient matrix of $p_2$ has rank one, while the coefficient
matrix of $p_1$ has rank two. Therefore, the equivalent optimal rank
one representation of system (\ref{pls-ex3}) is obtained for
\begin{multline*}
D_{p}=\begin{pmatrix}p_1 &0 &0\\ 0 &p_1 &0\\
0 &0 &p_2\end{pmatrix}, \; L=\begin{pmatrix}0 &0 &-1\\ 0 &1 &1\\
1 &1 &0\end{pmatrix}, \;
R=\begin{pmatrix}0 &1 &0\\ 0 &0 &1\\ 1 &-1 &0\end{pmatrix}, \\
 t=(2,-1,-1)^\top.
\end{multline*}
By Theorem \ref{pg-sol}, the parameterized solution is
\begin{multline}\label{ex3''}
x''(p,g)= \begin{pmatrix}0\\\frac{1}{3}\\\frac{1}{3}\end{pmatrix} +
\left(\check{A}^{-1}L\right)\begin{pmatrix}p_1 &0 &0\\ 0 &g &0\\
0 &0 &p_2\end{pmatrix} \begin{pmatrix}2.79556\\ 1.56338\\
2.249\end{pmatrix}, \\
 p_1,g\in [-\hat{p}_1,\hat{p}_1]={[}-\frac{1}{3},\frac{1}{3}{]}, \; p_2\in {[}-\frac{1}{2}, \frac{1}{2}{]}.
\end{multline}
Its interval evaluation $x''(\inum{p},\inum{g})$ is
$$\left([-1.032869, 1.032869], [-.795558, 1.462224], [.1032854, .5633813]\right)^\top.$$
For this example $x''(p,g)$ involves less number of interval
parameters than $x'(p,l)$, however $x_1''(\inum{p},\inum{g})\supset
x_1'(\inum{p},\inum{l})$. The latter disadvantage may vanish when
enclosing some secondary variables, see Example \ref{ex32}.
\end{example}

\section{Bounding secondary (derived) variables} \label{appl}

In this section we present a new application direction for the
parameterized solution enclosures and demonstrate the value of the
newly proposed parameterized solution. Before discussing practical
applications in structural mechanics, we make some general remarks.

\begin{remark}\label{rem1}
Since $x(\inum{p}_{\pi''}, \inum{g})$ in Theorem \ref{pg-sol} is
considered as a natural interval extension of the function $x(p)$,
$p\in [-\hat{p},\hat{p}]$ in (\ref{x(p')}), one has to be careful
when enclosing secondary variables $z(p, x(p_{\pi''}, g))$, where
$x(p_{\pi''}, g))$ is obtained by Theorem \ref{pg-sol}. In the
general case, the enclosure of the latter should be done by
computing the natural interval extension of $z(p, x(p))$, where
$x(p)$ is the function in (\ref{x(p')}), see Example \ref{kolevEx}.
\end{remark}

\begin{example}\label{kolevEx}
Consider a linear system which satisfies the requirements of
Corollary \ref{p-sol} and has a parameterized solution enclosure
\begin{equation}\label{pg-new}
x(p') = \check{x} + U p', \quad p'\in [-\hat{p}, \hat{p}],
\end{equation}
where $p'=(p_{\pi'}^\top, p_{\pi''}^\top)^\top$ and $U\in\R^{n\times
K}$ is the matrix that makes equivalent the above representation and
the representation (\ref{x(p)1}). With the additional requirements
$K=n$ and non-singularity of $U$, consider the secondary variables
\begin{equation}\label{z}
z(p') = U^{-1}x(p') - p', \quad  p'\in [-\hat{p}, \hat{p}].
\end{equation}
Replacing $x(p')$ form (\ref{pg-new}) in $z(p')$ we obtain
\begin{eqnarray}
z(p') &=& U^{-1}\check{x} + p' -p' \label{z1} \\
&=& U^{-1}\check{x}. \label{z2}
\end{eqnarray}
Due to the multiple occurrence of the parameters $p'$ in $z(p')$ and
according to Remark \ref{rem1}, interval enclosure of the variables
$z$ should be obtained as a natural interval extension of $z(p')$ in
(\ref{z}) or in (\ref{z1}) but not by $z(p')$ in (\ref{z2}), which
is not an interval function.
\end{example}

\begin{remark}
Let $z(p, x(p_{\pi''}, g))$, where $x(p_{\pi''}, g))$ is obtained by
Theorem \ref{pg-sol}, be $z(p, x(p_{\pi''}, g)) = B x(p_{\pi''}, g)
= B x(p)$, where $x(p)$ is the function in (\ref{x(p')}). Denote $V
= CLD_{|\inum{y}-t|}$. Then, the natural interval extension of
$$ B x(p) = B\check{x} + \left(BCF\right)p_{\pi''} + \left(BV\right)g(p_{\pi'})$$
is equal to the range of
$$B x(p_{\pi''}, g) = B\check{x} + \left(BCF\right)p_{\pi''} + \left(BV\right)g.$$
\end{remark}

\begin{example}\label{ex32}
Consider the system from Example \ref{Ex3} and the obtained two
parameterized solution enclosures $x'(p,l)$ and $x''(p,g)$. Find
interval enclosures of the secondary variables $z=B x$, where $x$ is
every one of the two parameterized solution enclosures, and
$B=\begin{pmatrix}1& 2& 3\\
3/2 & 1 & 2\\ 1/2 & 1/2 & 1\end{pmatrix}$. With $x'(p, l)$ in
(\ref{ex3'}) and the notation (\ref{x(p,l)}), we obtain
\begin{eqnarray*}
z\left(B, x'(p, l)\right) &=& B\tilde{x} + (BV)p' + (B\hat{l})l \\
&\in& \inum{z}' = \left([-1.2667226, 4.6000559], [-1.0233198,
3.02331978],
\right. \\
&& \hspace{2in} \left.  [-.38004820, 1.3800482]\right)^\top.
\end{eqnarray*}
With $x''(p, g)$ in (\ref{ex3''}) we obtain
\begin{multline*}
z\left(B, x''(p, g)\right) \in \inum{z}''=\left([-1.0222306,
4.3555640], \right.
\\ \left. [-.69090480, 2.6909048], [-.27465195,
1.2746520]\right)^\top.
\end{multline*}
The percentage by which $\inum{z}'$ overestimates\footnote{For
$\inum{x}'\supseteq \inum{x}''$, the percentage by which $\inum{x}'$
overestimates $\inum{x}''$ is $100(1- \hat{x}''/\hat{x}')$.}
$\inum{z}''$ is  $\left(8, 16, 12\right)^\top$\%.
\end{example}

In what follows we consider finite element (FE) models in linear
elastic structural mechanics involving interval uncertainties in
material and load parameters. The interval models are based on
classical interval arithmetic. While various methods and techniques
are devised for obtaining very sharp (even the exact) bounds for the
unknowns (e.g., displacements, called primary variables) of an
interval parametric linear system, obtaining sharp enclosure of the
so-called derived (secondary) variables (as axial forces, strains or
stresses) is referred in \cite{RaoMuMu11} as a challenging problem.
Secondary (derived) variables are functions of the primary variables
or of both primary variables and the initial interval model
parameters. Due to the dependency, the derived quantities are
obtained with significant overestimation. Some special techniques
are usually applied to decrease the overestimation in the secondary
quantities. In \cite{RaoMuMu11} a new mixed formulation of interval
finite element method (IFEM) is proposed, where both primary and
derived quantities of interest are involved as primary variables in
an expanded interval parametric  linear system. In this section we
propose an alternative approach based on the newly proposed
parameterized solution. The new approach requires that the interval
enclosure of the primary variables is obtained as a parameterized
solution. Thus, interval estimation of the secondary variables
reduces to range enclosure of the expressions representing secondary
variables as functions of the initial interval model parameters. In
formal notations the approach we propose, based on the new
parameterized solution of primary variables, is presented as
follows.

Let $A(p)u=a(p)$, $p\in\inum{p}$, be an interval parametric linear
system for the primary variables $u$ and $p\in\inum{p}$ be the
interval model parameters. For simplicity of the presentation we
assume that the coefficient matrices of all interval parameters have
rank one and the system for the primary variable can be solved by
Theorem \ref{solution1}. Depending on which material parameters are
considered as interval ones, an element secondary variable is
usually presented as $v = p_i b^\top u$, where $p_i$ is one of the
interval material parameters and $b$ is a numerical vector.
Algorithm \ref{algo} presents the new approach based on the newly
proposed parameterized enclosure of the displacements as primary
variables.

\begin{algorithm} \label{algo}
Interval enclosure of the secondary variable $v$ obtained by the new
parameterized enclosure (Corollary \ref{p-sol}) of the primary
variables $u$.
\medskip

\noindent {\bf Input}: \parbox[t]{4.7in}{numerical matrices
$A_0\in\R^{n\times n}$, $L,R^\top\in\R^{n\times K}$,
$F\in\R^{n\times (K-K_1)}$ and vectors $a_0\in\R^n$, $t\in\R^K$
providing an equivalent representation (\ref{plsG}); \\
vectors $b\in\R^n$ and $\inum{p}\in\IR^K$.}
\medskip

\noindent {\bf Output}: \parbox[t]{5in}{interval $\inum{v} = [v^-,
v^+]$ for the unknown secondary variable.}

\begin{enumerate}
\item {\bf Obtain} the new parameterized interval enclosure of the
primary variables by Corollary \ref{p-sol}
$$ u(p') = u_0 + U p', \quad p'\in [-\hat{p},\hat{p}], \qquad u_0\in\R^n, U\in\R^{n\times K}.$$
There is a flexibility in the implementation of this step of the
algorithm, which is discussed in \cite{PopovaED:2018}.

\item {\bf Generate} $\inum{p}'=[-\hat{p}, \hat{p}]$, \quad $\inum{v}' = \left(\check{p}_i+ \inum{p}_i'\right)\left(b^\top u_0 + (b^\top
U)\inum{p}'\right)$.

\item Since $v'(p')$ is a quadratic function of $p_i'$, $\inum{v}'$ may
overestimate the true range $\square\{v'(p')\mid p'\in\inum{p}'\}$.
To reduce the overestimation we may prove if this range is attained
at some endpoints of $\inum{p}_i'$. To this end we evaluate
$$
\frac{\partial v'(p')}{\partial p_i'} = b^\top u(p') + p_i b^\top
\frac{\partial u(p')}{\partial p_i'} \in \left(b^\top u_0 + (b^\top
U)\inum{p}'\right) + \inum{p}_i\left(b^\top U_{\bullet i}\right).
$$
{\bf Evaluate} $\inum{v}_1 = [v_1^-,v_1^+] = b^\top u_0 + (b^\top
U)\inum{p}'$ and $\inum{v}_2 = \inum{p}_i(b^\top U_{\bullet i})$.
\begin{enumerate}
\item[3.1]  {\bf If } $0\in v_1^-+\inum{v}_2$, {\bf then }
$v^-=(\inum{v}')^-$ \\
\hspace*{84pt} {\bf else} \; $s_1={\rm sign}(v_1^-+\inum{v}_2)\in\{-1,1\}$;\\
\hspace*{116pt} $(\inum{p}')_i = -s_1\hat{p}_i$; \\
\hspace*{116pt} $v^- = (\check{p}_i -s_1\hat{p}_i)(b^\top u_0 +
(b^\top U)\inum{p}')$;

\item[3.2] {\bf If } $0\in v_1^++\inum{v}_2$, {\bf then }
$v^+=(\inum{v}')^+$ \\
\hspace*{84pt} {\bf else} \; $s_2={\rm
sign}(v_1^++\inum{v}_2)\in\{-1,1\}$; \\
\hspace*{116pt} $(\inum{p}')_i = s_2\hat{p}_i$; \\
\hspace*{116pt} $v^+ = (\check{p}_i +s_1\hat{p}_i)(b^\top u_0 +
(b^\top U)\inum{p}')$;
\end{enumerate}

\item {\bf Return} $\inum{v}=[v^-, v^+]$.
\end{enumerate}
\end{algorithm}

\begin{lemma} \label{lemma}
Let $f(x):\inum{x}\subset\IR^m\rightarrow\IR^n$ be an interval
function such that each component
$f_i(x):\inum{x}\subset\IR^m\rightarrow\IR$ is a linear function of
all interval variables except one $x_{j_i}$ and the interval
variables $x_i$, $i\neq j_i$ appear only once in the expression of
$f_i(x)$. Then, for every $\tilde{x}\in\inum{x}$ and every $1\leq
i\leq n$ there exist $x_{j_i}', x_{j_i}''\in\inum{x}_{j_i}$ such
that
\begin{equation}\label{enclosure}
f_i(\inum{x}') \leq f_i(\tilde{x}) \leq f_i(\inum{x}''),
\end{equation}
where $\inum{x}'_i = \inum{x}_i'' = \inum{x}_i$, $i\neq j_i$,
$\inum{x}'_{j_i} = x_{j_i}'$, $\inum{x}''_{j_i} = x_{j_i}''$.
\end{lemma}

\begin{theorem} \label{algoProof}
Algorithm \ref{algo} provides interval enclosure of a secondary
variable $v$ with quality, which is not worse than the quality of
the enclosure of primary variables $u$ obtained by Corollary
\ref{p-sol} (Theorem \ref{solution1}) if steps 3.1 and 3.2 go to
their else-clause.
\end{theorem}
\begin{proof}
It is obvious that the natural interval extension
$\inum{v}'=v'(\inum{p}')$ of $v'(p')=$ \linebreak
$\left(\check{p}_i+ p_i'\right)\left(b^\top u_0 + (b^\top
U)p'\right)$ contains the true range of the secondary variable $v$.
Basing on Lemma \ref{lemma}, steps 3.1 and 3.2 of the algorithm try
to prove if some endpoint of $\inum{p}_i'$ generates the lower,
respectively the upper, endpoint of the true range of $v'(p')$.
Then, the sign of the intervals $v_1^-+\inum{v}_2$, respectively
$v_1^++\inum{v}_2$, will determine that. Let step $3$ of the
algorithm give the true range of $v'(p')$ on $p'\in\inum{p}'$, that
is $[v^-,v^+]=\square\{v'(p')\mid p'\in\inum{p}'\}$. We have that
$v'(p')$ is a function of $u(p')$ and $u(\inum{p}')$ is an outer
interval enclosure of $u(p')$, $p'\in\inum{p}'$, obtained by some
numerical method. Therefore, $[v^-,v^+]$ is an outer interval
enclosure of the secondary variable $v$ and the quality of this
enclosure depends of the quality of the enclosure $u(\inum{p}')$.
\end{proof}

For enclosing the derived variables considered in IFEMs of
mechanical structures it is essential that the parameterized
representation of the primary variables involves only the initial
interval parameters and their number is less than the number of
interval parameters in the parameterized solutions $x(p,l)$. In
Section \ref{ex6bar}, the parameterized solution obtained by Theorem
\ref{Kolev1} is representative of all parameterized solutions
$x(p,l)$ (\cite{Kol2014}--\cite{SkalnaHla19}) which satisfy Theorem
\ref{inclusion}.

\subsection{Truss Example 1}  \label{ex6bar}
Consider a $6$-bar truss structure as presented in Fig.
\ref{fig-6bar}  after \cite{QiuElish}. The structure consists of $6$
elements. The crisp values of the parameters of the truss are
presented in Table \ref{6bar-crispPars}.

\begin{figure}[h]
\centerline{\includegraphics[scale=.35]{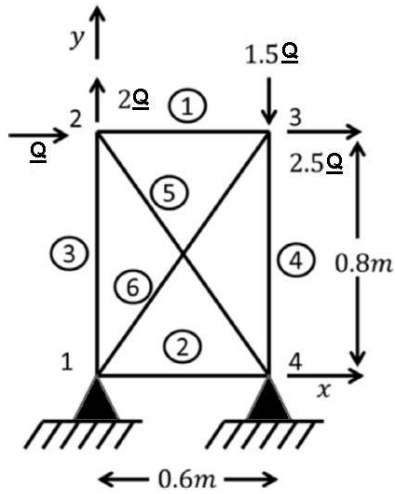}} \caption{A
$6$-bar truss structure after \cite{QiuElish}.}\label{fig-6bar}
\end{figure}

\begin{table}[h]
\begin{center}
\begin{tabular}{ll}
\hline Parameter & Value \\ \hline %
Modulus of elasticity for all elements \qquad \phantom{.} & \\ %
 $E_i$, $i=1,\ldots,6$ ($kN/m^2$) & $2.1\times 10^8$ \\ %
Cross sectional area &  \\ %
$A_1, A_2, A_3, A_4$ ($m^2$) & $1.0\times 10^{-3}$ \\ %
Cross sectional area &  \\ %
$A_5, A_6$ ($m^2$) & $1.05\times 10^{-3}$ \\[4pt] %
Load $Q$ ($kN$) & $20.5$ \\[4pt] %
Length of the first and second element &  \\ %
$L_1, L_2$ ($m$) & $0.6$ \\ %
Length of the third and fourth element &  \\ %
$L_3, L_4$ ($m$) & $0.8$ \\ %
Length of the fifth and sixth element &  \\ %
$L_6, L_6$ ($m$) & $1$ \\ %
\hline
\end{tabular}
\end{center}
\caption{Crisp values of the parameters for the  $6$-bar truss
structure.} \label{6bar-crispPars}
\end{table}

The traditional finite element method (FEM) for this structure leads
to a linear system
$$K(E,A,L)u = f(Q),$$
where $K(E,A,L)$ is the reduced stiffness matrix depending on the
structural parameters (modulus of elasticity $E$, cross sectional
area $A$, length $L$) for each element, $f(Q)$ is the load vector
and $u$ is the displacement vector. Namely,
\begin{multline} \label{Kfu}
K(E,A,L) \;= \\
\begin{pmatrix}\frac{E_1A_1}{L_1}+0.36\frac{E_5A_5}{L_5} & -0.48 \frac{E_5A_5}{L_5} & -\frac{E_1A_1}{L_1} & 0\\
-0.48 \frac{E_5A_5}{L_5} & \frac{E_3A_3}{L_3}+0.64\frac{E_5A_5}{L_5} & 0 & 0 \\ %
-\frac{E_1A_1}{L_1} &  0 & \frac{E_1A_1}{L_1}+0.36\frac{E_6A_6}{L_6} &  0.48\frac{E_6A_6}{L_6}\\ %
0 & 0 & 0.48\frac{E_6A_6}{L_6} & \frac{E_4A_4}{L_4}+0.64\frac{E_6A_6}{L_6} %
\end{pmatrix}, \\[6pt]
f(Q) \; = \; (Q, 2Q, 2.5Q, -1.5Q)^\top, \qquad\qquad %
u = (ux_2, uy_2, ux_3, uy_3)^\top.
\end{multline}

Let the load parameter $Q$ be unknown-but-bounded in the interval
$\inum{Q}=[20, 21] kN$ and the cross sectional areas $A_5$, $A_6$ be
also uncertain varying in the intervals $[1.008, 1.092]\times
10^{-3}$ $m^2$, $[1,1.1]\times 10^{-3}$ $m^2$, respectively. %
The aim is to obtain interval enclosure for the displacements (as
primary variables depending on interval model parameters) and for
the  element axial forces (as secondary variables). Axial forces are
quantities of practical interest in design. For the considered
example, the global force vector $F=(F_{e_1}, F_{e_3}, F_{e_4},
F_{e_5}, F_{e_6})^\top$ is determined by $F = D_vTu$, where
$F_{e_i}$ are the corresponding element\footnote{Finite elements are
denoted by $e_i$.} forces and
$$
v=\left(\frac{E_1A_1}{L_1}, \frac{E_3A_3}{L_3}, \frac{E_4A_4}{L_4},
\frac{E_5A_5}{L_5}, \frac{E_6A_6}{L_6}\right)^\top,  \quad T
=\begin{pmatrix}-1 & 0 &1 & 0\\0 &1 &0 &0\\0 &0 &0 &1\\-\frac{6}{10}
&\frac{8}{10} &0 &0\\0 &0 & \frac{8}{10} &\frac{6}{10}
\end{pmatrix}.
$$
Above, the displacements $u=u(A_5,A_6,Q)$ (as primary variables),
the vector $v=v(A_5,A_6)$ and the secondary variables -- element
axial forces $F_{e_i}$, $i=1,3,4,5,6$  -- are functions of the
interval model parameters $A_5,A_6,Q$.

First, we find interval enclosures for the displacements as
parameterized solutions to the interval parametric linear system
$K(A_5,A_6)u = f(Q)$. Applying Theorem \ref{Kolev1} we obtain
\begin{multline*}
10^4 u'(A_5,A_6,Q, l) \approx \\ \begin{pmatrix} %
8.5846 -2153.0A_5 - 2134.2A_6 +  .41896Q + 10^{-3}26.693 l_1\\ %
3.2669 + 491.791A_5 - 559.409A_6 + .15937Q + 10^{-3} 6.6039 l_2 \\ %
8.9579 -1876.6 A_5 - 2449.5 A_6 + .43727Q + 10^{-3} 26.952 l_3\\ %
-3.1109 + 491.80 A_5 - 559.420A_6 + .15176Q + 10^{-3} 6.6017 l_4 
\end{pmatrix},
\end{multline*}
where $A_i\in [-\hat{A}_i, \hat{A}_i]$, $i=5,6$, $Q\in [-\hat{Q},
\hat{Q}]$, $l_i\in [-1,1]$, $i=1,\ldots,4$. Its interval hull is
\begin{multline*}
10^4 u'(\inum{A}_5,\inum{A}_6,\inum{Q}, \inum{l})  \subset \\
\left([8.151, 9.018], [3.131, 3.402], [8.511, 9.405], [-3.242,
-2.979]\right)^\top.
\end{multline*}

With the crisp values from Table \ref{6bar-crispPars}, the optimal
equivalent rank one representation of the system (\ref{Kfu}) is
$(A_0+LD_gR)u = f(Q)$, where $g=(A_5,A_6)^\top$ and
$$
10^{-5}A_0=\begin{pmatrix}\frac{7}{2}&0&-\frac{7}{2}&0\\
0&\frac{21}{8}&0&0\\-\frac{7}{2}&0&\frac{7}{2}&0\\0&0&0&\frac{21}{8}\end{pmatrix},
\;
10^{-5}L=\begin{pmatrix}756&0\\-1008&0\\0&756\\0&1008\end{pmatrix},\;
R^\top=\begin{pmatrix}1&0\\-\frac{4}{3}&0\\
0&1\\0&\frac{4}{3}\end{pmatrix}.
$$
The application of Corollary \ref{p-sol} to the above system yields
the parameterized solution
\begin{multline*}
10^4 u''(A_5,A_6,Q) \approx \begin{pmatrix} %
8.5846 + 2306.60 A_5 + 2285.45 A_6 - .41876 Q \\ %
3.2669 - 527.104 A_5 + 599.307 A_6 - .15936 Q \\ %
8.9579 + 2010.11 A_5 + 2622.56 A_6 - .43697 Q \\%
-3.1109 -527.104 A_5 + 599.307 A_6 + .15175 Q
\end{pmatrix},
\end{multline*}
where $A_i\in [-\hat{A}_i, \hat{A}_i]$, $i=5,6$, $Q\in [-\hat{Q},
\hat{Q}]$. Its interval hull is
\begin{multline*}10^4 u''(\inum{A}_5,\inum{A}_6,\inum{Q}) \subset
\left([8.164, 9.006], [3.135, 3.399], \right. \\ \left. [8.523,
9.392], [-3.239, -2.982]\right)^\top.
\end{multline*}
It is readily seen that $u''(\inum{A}_5,\inum{A}_6,\inum{Q})$
provides sharper interval enclosure to the displacements than
$u'(\inum{A}_5,\inum{A}_6,\inum{Q},\inum{l})$. Percentage by which
the latter overestimates the former is $(2.95, 2.32, 2.87,
2.39)^\top$. This implies that the newly proposed parameterized
solution $u''(A_5,A_6,Q)$ will provide a sharper enclosure of the
element axial forces.

For the particular example we have
\begin{eqnarray*}
10^{-5}D_vT &=& 10^{-5}D_{v'}T' \\
&=& 10^{-5}D_{v'}\begin{pmatrix}-\frac{7}{2}&0&\frac{7}{2}&0\\ 0&\frac{21}{8}&0&0\\
0&0&0&\frac{21}{8}\\
-1260 & 1680 & 0&0\\
 0 &0 &1680  & 1260
\end{pmatrix}, \; v'=\begin{pmatrix}1\\1\\1\\A_5\\A_6\end{pmatrix},
\end{eqnarray*}
which shows that the element axial forces  $F_{e_i}$, $i=1,3,4$, are
linear functions of the interval model parameters $A_5,A_6,Q$, while
the axial forces  $F_{e_5}, F_{e_6}$ are quadratic functions of the
interval parameters $A_5,A_6$, respectively. In what follows, in
particular Table \ref{forceCompar}, we use the notation
\begin{equation*}
\begin{split} & u'(A_5,A_6,Q,l) = x'_0 +B'(A_5,A_6,Q)^\top + C'l,
\quad
l\in\inum{l}=([-1,1],\ldots,[-1,1])^\top, \\
&u''(A_5,A_6,Q) = x''_0 +B''(A_5,A_6)^\top + C''(Q),\\
& \inum{u}'= u'(\inum{A}_5,\inum{A}_6,\inum{Q},\inum{l}), \qquad
\inum{u}''= u''(\inum{A}_5,\inum{A}_6,\inum{Q}),\\
& \inum{F}' = D_{\inum{v}'}(T'\inum{u}'), \qquad \inum{F}'' =
D_{\inum{v}'}(T'\inum{u}''), \qquad\qquad [\mp\hat{a}] = [-\hat{a},\hat{a}], \\
& F'(\inum{A}_5,\inum{A}_6,\inum{Q},\inum{l}) =
D_{\inum{v}'}\left(T'x'_0 +
(T'B')([\mp\hat{A}_5], [\mp\hat{A}_6], [\mp\hat{Q}])^\top + T'\inum{l}\right), \\ %
& F''(\inum{A}_5,\inum{A}_6,\inum{Q}) = D_{\inum{v}'}\left(T'x''_0
 + (T'B'')([\mp\hat{A}_5], [\mp\hat{A}_6])^\top +
 (T'C'')([\mp\hat{Q}])\right),
\end{split}
\end{equation*}
where $x'_0$, $B'$, $C'$ are numerical vector and matrices presented
in $u'(A_5,A_6,Q,l)$ above, and  $x''_0$, $B''$, $C''$ are numerical
vector and matrices presented in the numerical expression of
$u''(A_5,A_6,Q,l)$ above. Table \ref{forceCompar} presents and
compares interval enclosures of the element axial forces
$\inum{F}',\inum{F}''$, obtained via direct interval computation,
and the enclosures $F'(\inum{A}_5,\inum{A}_6,\inum{Q},\inum{l})$,
$F''(\inum{A}_5,\inum{A}_6,\inum{Q})$, obtained via the two kinds
parameterized solutions $u'(A_5,A_6,Q,l)$ and $u''(A_5,A_6,Q)$,
respectively.
\begin{table}[h]
\begin{center}
\begin{tabular}{llll}
\hline
 & $\inum{F}'=$ & $\inum{F}''$ & $F''(\inum{A}_5,\inum{A}_6,\inum{Q})$ \\
 & $F'(\inum{A}_5,\inum{A}_6,\inum{Q},\inum{l})$ & \\
\hline $e_1$ & [-17.740, 43.875] & [-16.843, 42.978]& [11.722,
14.412] \\
$e_3$ & [82.215, 89.298] & [82.297, 89.216] &[82.297, 89.216]  \\
$e_4$ & [-85.102, -78.218] & [-85.020, -78.300]&[-85.019, -78.300]
\\
$e_5$ & [-66.621, -45.919] & [-66.388, -46.135]& [-62.365, -49.848]
\\
$e_6$ & [102.13, 132.51] & [102.39, 132.23] & [104.86, 129.51] \\
\end{tabular}
\end{center}
\caption{Interval enclosures for the element axial forces in the
$6$-bar truss structure obtained via the two kinds of parameterized
solutions.} \label{forceCompar}
\end{table}

Due to $\inum{u}''\subset\inum{u}'$, it is clear that
$\inum{F}''\subset\inum{F}'$ and the latter overestimation is $(2.9,
2.3, 2.4, 2.2, 1.8)^\top$\%. Note that the enclosures $\inum{F}',
\inum{F}''$ are so bad that the sign of $F_{e_1}$ cannot be
determined. Note also that
$F'(\inum{A}_5,\inum{A}_6,\inum{Q},\inum{l})=\inum{F}'$. The latter
means that Kolev-style parameterized solution was not able to
improve the bounds $\inum{F}'$. Intervals $\inum{F}''$ overestimate
intervals $F''(\inum{A}_5,\inum{A}_6,\inum{Q})$ by $(95.5, 0, 0,
38.2, 17.4)^\top$\%, respectively. Since $F_{e_i}(A_5,A_6,Q)$,
$i=5,6$, are quadratic polynomials of the interval parameters $A_5$,
$A_6$, (represented by a $v'$ in Algorithm \ref{algo}) respectively,
their interval values presented in Table \ref{forceCompar}, in
general, may not be equal to the corresponding ranges. Evaluating
partial derivatives as presented in Algorithm \ref{algo}, we prove
that $F_{e_5}''(A_5,A_6,Q)$ is monotonic decreasing on $A_5$, while
$F_{e_6}''(A_5,A_6,Q)$ is monotonic increasing on $A_6$. Thus
$F''(\inum{A}_5,\inum{A}_6,\inum{Q})$ presented in Table
\ref{forceCompar} are the exact ranges of the corresponding
expressions of $v'$ and the quality of the enclosures
$F''(\inum{A}_5,\inum{A}_6,\inum{Q})$  is the same as the quality of
the enclosures $\inum{u}''$. Note that neither $\inum{u}''$ nor
$F''(\inum{A}_5,\inum{A}_6,\inum{Q})$ are the exact ranges of the
corresponding unknowns, see the proof of Theorem \ref{algoProof}. In
order to demonstrate the quality of the enclosures
$F''(\inum{A}_5,\inum{A}_6,\inum{Q})$ we give below the
corresponding exact ranges rounded outwardly
\begin{multline*}
F \in \left([11.8215, 14.3755],[82.4287, 89.1673],[-84.9499,
-78.4121],\right. \\
\left. [-58.9591, -53.0358],[109.960, 123.970]\right)^\top.
\end{multline*}

\subsection{Truss Example 2}

Consider a finite element model of a one-bay $20$-floor truss
cantilever presented in Fig. \ref{bench2}, after \cite{Muhanna04}.
\begin{figure}[ht]
\centerline{\includegraphics[height=2.5in]{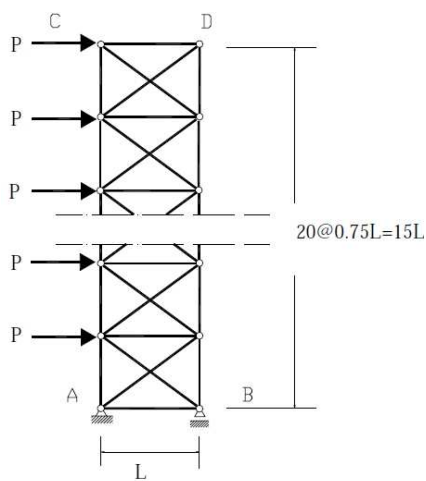}}
\caption{One-bay $20$-floor truss cantilever after
\cite{Muhanna04}.} \label{bench2}
\end{figure}
The structure consists of 42 nodes and 101 elements. The bay is $L =
1$m, every floor is $0.75 L$, the element cross-sectional area is $A
= 0.01$ m$^2$, and the crisp value for the element Young modulus is
$E = 2\times 10^8$kN/m$^2$. Twenty horizontal loads with nominal
value $P=10$ kN are applied at the left nodes. The boundary
conditions are determined by the supports: at $A$ the support is a
pin, at $B$ the support is roller. It is assumed $10$\% uncertainty
in the modulus of elasticity $E_k$ of each element ($\mp 5$\% from
the corresponding mean value) and $10$\% uncertainty in the twenty
loads. The goal is to obtain bounds for the axial force ($F_{40}$)
in element $40$.

Exactly this problem is used in \cite{RaoMuMu11} as a benchmark
problem for the applicability, computational efficiency and
scalability of the approach proposed therein for structures with
complex configuration and a large number of interval parameters. The
aim of using this example in the present work is similar: to check
these properties for the newly proposed Algorithm \ref{algo} based
on the new  parameterized solution. In addition, the interval result
obtained by the approach proposed here will be compared to the
results obtained by various other approaches considered in
\cite[Example 2]{RaoMuMu11}.

Table \ref{trussCant} presents intervals for the axial force
$F_{40}$ in element $40$, which are obtained by:
\begin{itemize}
\item the special expanded finite element formulation, proposed in
\cite{RaoMuMu11}, ($\inum{F}_{40}$);
\item the newly proposed parameterized solution and step 2 of Algorithm \ref{algo}, ($F'_{40}(\inum{E},\inum{P})$);
\item the newly proposed parameterized solution and step 3 of Algorithm
\ref{algo}, ($F''_{40}(\inum{E},\inum{P})$).
\end{itemize}

\begin{table}[h]
\begin{center}
\begin{tabular}{lll}
\hline
$\inum{F}_{40}$ by \cite{RaoMuMu11} & $F'_{40}(\inum{E},\inum{P})$ & $F''_{40}(\inum{E},\inum{P})$ \\ %
\hline %
 [60.652, 98.991] \quad & [55.729, 106.03] \quad & [61.595, 98.639] \\ %
\end{tabular}
\end{center}
\caption{Axial force $F_{40}$ (kN) in element $40$ of the cantilever
truss obtained by various approaches.} \label{trussCant}
\end{table}

Interval values for the axial force $F_{40}$, obtained by Pownuk's
``gradient-free" method \cite{Pownuk09} and by the Neumaier's
enclosure $\inum{z}_2(u)$ \cite[Eqn. (4.13)]{NePow}, are presented
in \cite{RaoMuMu11} and can be compared.

It should be mentioned that the coefficient matrices of all interval
parameters in the linear system for the displacements have rank one.
Therefore, there are no exceed interval parameters in the
parameterized solution enclosure for the displacements. The symbolic
expression of $F_{40}(E,P)$ is a quadratic function of the interval
parameter $E_{40}$. Applying step 3 of Algorithm \ref{algo} we prove
numerically that both the lower and the upper bounds of
$F_{40}(E,P)$ are attained at the upper bound of $\inum{E}_{40}$.
Note that this does not mean monotonic dependence of
$\inum{E}_{40}$. Note also that the above proof is very easy
compared to proving monotonic dependence of the displacements on the
interval parameters. Step 3 in Algorithm \ref{algo} costs nothing
compared to step 2 of the algorithm. Thus, we obtain an improvement
$F''_{40}(\inum{E},\inum{P})$ of the bounds
$F'_{40}(\inum{E},\inum{P})$  in Table \ref{trussCant}. Interval
$F''_{40}(\inum{E},\inum{P})$  is sharper than the interval
$\inum{F}_{40}$, obtained by the approach of \cite{RaoMuMu11}, which
shows the efficiency of the newly proposed approach based on the new
parameterized solution. It should be also mentioned that the
interval axial force $\inum{z}_2(u)$, showed in \cite[Table
4]{RaoMuMu11} and obtained by the Neumaier's approach \cite[Eqn.
(4.13)]{NePow}, is the same as the interval
$F'_{40}(\inum{E},\inum{P})$ in Table \ref{trussCant}.

\section{Conclusion}

We presented a new kind of parameterized solution to interval
parametric linear systems. It is based on optimal rank one
representation of the parameter dependencies. This representation
determines the number of interval parameters in the parameterized
solution, as well as, whether the new parameterized solution will
have better properties than the Kolev-style parameterized solutions.
The new parameterized solution possesses similar computational
complexity as the corresponding numerical method of Theorem
\ref{solution1}. Comparing to some parameterized solutions based on
the weaker condition (\ref{rhoDelta}), the price we pay for a better
solution depends on the difference between a matrix inversion,
$O(s^3)$, when solving equation (\ref{eqY}) of dimension $s$ and the
matrix inversion, $O(n^3)$, related to most of the other methods.
Note that in practical applications, where the new solution is most
efficient, the dependencies have rank one structure and $s=K_1$. The
newly proposed parameterized solution does not require affine
arithmetic as most of the parameterized solutions considered in
\cite{Skalna}. The required rank one representation of the parameter
dependencies is usually inherited by the corresponding model in an
application domain, see, e.g., \cite{NePow}. Otherwise, one can use
the matrix full rank factorization utilities provided by many
general purpose computing environments.

The major advantage of the newly proposed parameterized solution is
for  interval parametric linear systems involving rank one
uncertainty structure. Such systems appear often in various
domain-specific models, cf., \cite{Pop2014CAMWA}. A general
application direction is presented in this article and illustrated
by some numerical examples originated from worst-case analysis of
truss structures in mechanics.

While bounding secondary variables by the approach of
\cite{RaoMuMu11} requires a dedicated IFEM formulation for each
particular problem and the system to be solved is expanded by the
number of derived quantities, the approach based on the new
parameterized solution of primary variables does not depend on the
IFEM formulation, does not require solving an expanded interval
parametric linear system, and provides sharp bounds for the derived
quantities  by a simple interval evaluation. The proposed new
approach could be applied for enclosing secondary variables in
various other domains where the uncertainties have rank one
structure.

\section*{Acknowledgements}
This work is supported by the Grant No BG05M2OP001-1.001-0003,
financed by the Bulgarian Operational Programme ``Science and
Education for Smart Growth"  (2014-2020) and co-financed by the
European Union through the European structural and investment funds.



\begin{thebibliography}{00}

\bibitem{Kol2014} L. Kolev, Parametrized solution of linear interval parametric
systems, Appl. Math. Comput. 246 (2014) 229--246.
\url{https://doi.org/10.1016/j.amc.2014.08.037}.

\bibitem{Kol2016a} L. Kolev, A direct method for determining a $p$-solution of linear parametric
systems, J. Appl. Computat. Math. 5 (2016) 1--5.
\url{https://doi.org/10.4172/2168-9679.1000294}.

\bibitem{Kol2016b} L. Kolev, A class of iterative methods for determining p-solutions of linear interval parametric
 systems, Rel. Comput. 22 (2016) 26--46.

\bibitem{Kolev2018} L. Kolev, P-solutions for a class of structured interval parametric systems, Preprint in Research Gate,
23 November 2018, submitted as the file ``p-sol\_struct\_correct".
\url{https://doi.org/10.13140/RG.2.2.14958.25921}.

\bibitem{Skalna} I. Skalna, Parametric Interval Algebraic Systems,
Studies in Computational Intelligence 766, Springer (2018).

\bibitem{SkalnaHla19} I. Skalna, M. Hlad\'ik, Direct and iterative methods for interval parametric algebraic systems producing parametric
solutions, Numer. Linear Algebra Appl. (2019) e2229.
\url{https://doi.org/10.1002/nla.2229}.


\bibitem{PopovaED:2017:POE} E.D. Popova, Parameterized outer estimation of AE-solution sets to parametric interval linear
systems, Appl. Math. Comput. 311 (2017) 353--360.
\url{https://doi.org/10.1016/j.amc.2017.05.042}.

\bibitem{PopovaED:2018} E.D. Popova, Rank one interval enclosure of the parametric united solution
set, BIT Numer. Math. 59 (2) (2019) 503--521.
\url{https://doi.org/10.1007/s10543-018-0739-4}.

\bibitem{Rohn2011v22} J. Rohn, Explicit inverse of an interval matrix with unit
midpoint, Electronic Journal of Linear Algebra 22 (2011) 138--150.
\url{https://doi.org/10.13001/1081-3810.1430}.

\bibitem{Piziak} R. Piziak, P.L. Odell, Full rank factorization of
matrices, Mathematics Magazine 72 (1999) 193--201.

\bibitem{PopovaED:2017:ESS} E.D. Popova, Enclosing the solution set of parametric interval matrix equation
$A(p)X=B(p)$, Numer. Algor. 78 (2018) 423--447.
\url{https://doi.org/10.1007/s11075-017-0382-1}.

\bibitem{Pop2014CAMWA} E.D. Popova, Improved enclosure for some parametric solution sets with linear
shape, Computers and Mathematics with Applications 68 (2014)
994--1005. \url{https://doi.org/10.1016/j.camwa.2014.04.005}.

\bibitem{NePow} A. Neumaier, A. Pownuk, Linear systems with large uncertainties, with applications to truss
structures, Rel. Comput. 13 (2007) 149--172.
\url{https://doi.org/10.1007/s11155-006-9026-1}.

\bibitem{RaoMuMu11} M.V. Rama Rao, R.L. Mullen, R.L. Muhanna, A new interval finite element formulation with the same accuracy in primary and derived
variables, Int. J. Reliability and Safety 5 (2011) 336--357.
\url{https://doi.org/10.1504/IJRS.2011.041184}.

\bibitem{Kearfott} R. B. Kearfott, Rigorous Global Search:
Continuous Problems, Kluwer Academic Publishers (1996)

\bibitem{QiuElish} Z. Qiu, I. Elishakoff, Antioptimization of structures with large uncertain-but-non-random parameters via interval
analysis, Computer Methods in Applied Mechanics and Engineering 152
(1998) 361--372.
\url{https://doi.org/10.1016/S0045-7825(96)01211-X}.

\bibitem{Muhanna04}  R.L. Muhanna, Benchmarks for interval finite element
computations, Center for Reliable Engineering Computing, Georgia
Tech, USA, \url{http://rec.ce.gatech.edu/resources/Benchmark_2.pdf},
2004 (accessed 12 November 2018).

\bibitem{Pownuk09} A. Pownuk, N.K.G. Ramunigari, Design of 2D elastic structures with the interval
parameters, Proceedings of 11th WSEAS International Mathematical and
Computational Methods in Science and Engineering (MACMESE'09),
Baltimore, MD, USA, 25-29 November 2009.
\url{http://www.wseas.us/e-library/conferences/2009/baltimore/MACMESE/MACMESE-01.pdf}.

\end{thebibliography}
\end{document}